\documentclass[a4paper,11pt]{amsart}

\usepackage{amsmath, amsthm, amsfonts, amssymb}
\usepackage[all,cmtip]{xy}

\usepackage{enumerate}
\usepackage[colorlinks=true, linkcolor=blue, citecolor=red, menucolor=black]{hyperref}

\usepackage{color} 
\usepackage{graphicx}
\usepackage{mathrsfs}

\numberwithin{equation}{section}

\newtheorem{letterthm}{Theorem}

\newtheorem{theorem}{Theorem}[section]
\newtheorem{lemma}[theorem]{Lemma}
\newtheorem{corollary}[theorem]{Corollary}
\newtheorem{proposition}[theorem]{Proposition}

\theoremstyle{definition}
\newtheorem{remark}[theorem]{Remark}

\newtheorem{example}[theorem]{Example}

\newtheorem{definition}[theorem]{Definition}

\newtheorem*{definition*}{Definition}
\newtheorem*{conjecture}{Conjecture}
\newtheorem*{definitions*}{Definitions}

\newtheorem*{claim*}{Claim}
\newtheorem*{question*}{Question}

\newcommand{\R}{\mathbb{R}}
\newcommand{\B}{\mathrm{B}}
\newcommand{\C}{\mathbb{C}}
\newcommand{\Z}{\mathbb{Z}}

\newcommand{\Q}{\mathbb{Q}}
\newcommand{\N}{\mathbb{N}}
\newcommand{\T}{\mathbb{T}}

\newcommand{\cZ}{\mathcal{Z}}
\newcommand{\cU}{\mathcal{U}}

\newcommand{\cP}{\mathcal{P}}

\newcommand{\cS}{\mathcal{S}}
\newcommand{\cT}{\mathcal{T}}

\newcommand{\Ad}{\operatorname{Ad}}
\newcommand{\id}{\text{\rm id}}
\newcommand{\ri}{\text{\rm i}}
\newcommand{\rd}{\text{\rm d}}

\newcommand{\Aut}{\operatorname{Aut}}

\newcommand{\Inn}{\operatorname{Inn}}

\newcommand{\rB}{\mathord{\text{\rm B}}}
\newcommand{\rZ}{\mathord{\text{\rm Z}}}

\newcommand{\rL}{\mathord{\text{\rm L}}}
\newcommand{\rC}{\mathord{\text{\rm C}}}
\newcommand{\rM}{\mathord{\text{\rm M}}}

\newcommand{\rE}{\mathord{\text{\rm E}}}

\newcommand{\ovt}{\mathbin{\overline{\otimes}}}
\newcommand{\bigovt}{\mathbin{\overline{\bigotimes}}}

\newcommand{\op}{\mathord{\text{\rm op}}}

\newcommand{\cR}{\mathcal{R}}

\newcommand{\I}{{\rm I}}
\newcommand{\II}{{\rm II}}
\newcommand{\III}{{\rm III}}

\allowdisplaybreaks

\makeatletter
\@namedef{subjclassname@2020}{%
  \textup{2020} Mathematics Subject Classification}
\makeatother

\begin{document}

\title[Uniqueness of almost periodic outer flows]{Uniqueness of almost periodic outer flows on the hyperfinite type $\II_1$ factor}

\begin{abstract}
We show that any almost periodic outer flow $\alpha : \R \curvearrowright R$ on the hyperfinite type $\II_1$ factor with Connes'\! spectrum $\Gamma(\alpha) = \R$ satisfies the Rokhlin property and thus is unique up to cocycle conjugacy. The proof relies on a key cocycle perturbation result for type $\III$ amenable equivalence relations. As a byproduct of our methods, we also show that every almost periodic factor of type $\III_1$ with separable predual has an extremal almost periodic faithful normal state.
\end{abstract}

\author{Cyril Houdayer}
\address{\'Ecole normale sup\'erieure \\ D\'epartement de math\'ematiques et applications \\ Universit\'e Paris-Saclay \\ 45 rue d'Ulm \\ 75230 Paris Cedex 05 \\ France}
\email{cyril.houdayer@ens.psl.eu}
\thanks{CH is supported by ERC Advanced Grant NET 101141693}

\author{Amine Marrakchi}
\address{CNRS \\ \'Ecole normale sup\'erieure de Lyon \\ Unit\'e de math\'ematiques pures et appliquées  \\ 46 all\'ee d'Italie \\ 69364 Lyon \\ France}
\email{amine.marrakchi@ens-lyon.fr}


\maketitle

\section{Introduction and statement of the main results}

\subsection{Uniqueness of almost periodic outer flows on $R$}

By the fundamental work of Connes \cite{Co75b}, there is a unique amenable type $\II_1$ factor with separable predual: it is isomorphic to the unique hyperfinite type $\II_1$ factor $R$ of Murray--von Neumann \cite{MvN43}. Moreover, Connes \cite{Co75a} showed that there is a unique automorphism $\theta \in \Aut(R)$ up to outer conjugacy when the corresponding action $\Z \curvearrowright R$ is outer, i.e.\! $\theta^n \notin \Inn(R)$ for every $n \in \Z \setminus \{0\}$. 

A next natural step is to classify continuous flows $\R \curvearrowright R$ on the hyperfinite type $\II_1$ factor. This problem, which has been pushed forth by Takesaki since the early 1980s, is still open. In order to state the classification problem for flows $\R \curvearrowright R$ on the hyperfinite type $\II_1$ factor, we follow Kawahigashi's beautiful exposition on Popa's W$^*$-news blog \cite{Po25}.

A natural replacement of the outer conjugacy for a single automorphism is the cocycle conjugacy for a flow. However, it is not straightforward to define the right notion of ``outerness" for a flow. We should keep in mind that there is a certain similarity between the classification of flows on the hyperfinite type $\II_1$ factor up to cocycle conjugacy and the classification of amenable type $\III$ factors and their modular automorphism group. In that respect, an important invariant for the cocycle conjugacy class of a flow $\alpha : \R \curvearrowright R$ is Connes'\! spectrum $\Gamma(\alpha)$, which is a closed subgroup of $\R$. The condition $\Gamma(\alpha) = \R$ is equivalent to the factoriality of the crossed product von Neumann algebra $R \rtimes_\alpha \R$. However, unlike the analogous situation of the modular automorphism group $\sigma^\varphi$ of a type $\III_1$ factor $M$, the condition $\Gamma(\alpha) = \R$ does not necessarily imply that $\alpha$ is \emph{outer}, i.e.\! that $\alpha_t \notin \Inn(R)$ for every $t \in \R \setminus \{0\}$. A stronger assumption that one can make on $\alpha$ is \emph{strict outerness} which means that $R' \cap (R \rtimes_\alpha \R)=\C1$. If $\alpha$ is strictly outer, then $\Gamma(\alpha)=\R$ and $\alpha$ is outer. Note that for a type $\III_1$ factor $M$, the modular automorphism group $\sigma^\varphi$ is always strictly outer by Connes--Takesaki relative commutant theorem \cite{CT76}.

As in the classification of amenable type $\rm III_\lambda$ factors with $\lambda\in [0, 1)$ \cite{Co72, Co75b}, Kawahigashi \cite{Ka87} obtained a complete classification of flows $\alpha : \R \curvearrowright R$ up to stable conjugacy in the case when $\Gamma(\alpha)\neq\R$. Similar to the uniqueness of the amenable type ${\rm III_1}$ factor \cite{Co85, Ha85}, it is natural to expect that there is a unique flow $\alpha : \R \curvearrowright R$ up to cocycle conjugacy in the case when $\Gamma(\alpha) =\R$ and $\alpha$ is outer.  Kawahigashi \cite{Ka88a} settled this question affirmatively under the extra assumption that $\alpha : \R \curvearrowright R$ fixes pointwise a Cartan subalgebra $A \subset R$.

More recently, Masuda--Tomatsu \cite{MT12} developed a conceptual approach to classifying flows on von Neumann algebras up to cocycle conjugacy. Notably, they showed that a flow $\alpha : \R \curvearrowright R$ on the hyperfinite type $\II_1$ factor which satisfies the Rokhlin property is unique up to cocycle conjugacy. This enabled them to obtain alternative proofs of Kawahigashi's classification results. Following \cite{Ki95, Ka00, MT12}, we say that the flow $\alpha : \R \curvearrowright R$ satisfies the {\em Rokhlin property} if for every $p \in \R$, there exists a unitary $u_p \in \mathcal U(R_{\alpha, \omega})$ in the $(\alpha, \omega)$-equicontinuous part of the central sequence algebra such that $\alpha^\omega_t(u_p) = \exp({\rm i}tp) u_p$ for every $t \in \R$.

To state the conjecture explicitly, we let $\sigma^\infty = \bigovt_{\N}  \sigma$ be the infinite tensor product action on $R=\bigovt_{\N} \rM_n(\C)$ where $\sigma : \R \curvearrowright \rM_n(\C)$ is any faithful flow on a finite dimensional factor, for example
$$ \sigma : t \mapsto \Ad \left( \begin{array}{ccc} 1 & 0 &0 \\ 0 & \exp({\rm i}t\lambda) &0 \\ 0 & 0 & \exp({\rm i}t\mu) \end{array} \right) \in \Aut(\rM_3(\C))
$$
with $\lambda, \mu > 0$ and $\lambda/\mu \notin \Q$. The following conjecture mentioned by Kawahigashi in \cite{Po25} describes the main open problem on the classification of continuous flows on the hyperfinite type $\II_1$ factor.

\begin{conjecture}
Let $\alpha : \R \curvearrowright R$ be a flow on the hyperfinite type $\II_1$ factor. Then the following assertions are equivalent:
\begin{enumerate}
    \item $\alpha$ is cocycle conjugate to $\sigma^{\infty}$.
    \item $\alpha$ satisfies the Rokhlin property.
    \item $\alpha$ is strictly outer, i.e.\ $R' \cap (R \rtimes_\alpha \R)=\C1$.
    \item $\Gamma(\alpha) = \R$ and $\alpha$ is outer.
\end{enumerate}
\end{conjecture}
As we explained above, for any flow $\alpha : \R \curvearrowright R$, we have $(1) \Leftrightarrow (2)$ by \cite{MT12} and moreover we have $(2) \Rightarrow (3) \Rightarrow (4)$.

In this paper, we settle the above conjecture for the class of {\em almost periodic} flows $\alpha : \R \curvearrowright R$. A flow $\alpha : \R \curvearrowright R$ is almost periodic if the closure $K = \overline{\left\{\alpha_t \mid t \in \R \right\}} < \Aut(R)$ is a compact subgroup. The class of almost periodic flows $\alpha : \R \curvearrowright R$ is quite large and contains all flows arising from higher dimensional simple noncommutative tori \cite{Ki95, Ka88b}. 

\begin{letterthm} \label{main almost periodic flow}
Let $\alpha : \R \curvearrowright R$ be an almost periodic flow such that $\Gamma(\alpha) = \R$ and $\alpha$ is outer. Then $\alpha$ has the Rokhlin property, hence it is cocycle conjugate to $\sigma^\infty$.
 \end{letterthm}

We emphasize the fact that Theorem \ref{main almost periodic flow} holds for {\em arbitrary} almost periodic flows $\alpha : \R \curvearrowright R$ with no extra assumptions on the fixed-point subalgebra $R^\alpha \subset R$. Prior to our work, the uniqueness of almost periodic flows on the hyperfinite type $\II_1$ factor up to cocycle conjugacy was only known in some very specific cases such as minimal almost periodic flows (i.e.\! $(R^\alpha)' \cap R = \C1$) \cite[Theorem 6.12]{MT12}. In \cite{Ka88b}, an interesting example of an ergodic almost periodic flow (i.e.\! $R^\alpha = \C1$) is also treated.

We also point out that Theorem \ref{main almost periodic flow} holds more generally for almost periodic actions of arbitrary second countable locally compact abelian groups on the hyperfinite $\II_1$ factor (see Theorem \ref{almost periodic action} below).

The proof of Theorem \ref{main almost periodic flow} divides into two independent steps. First, we prove the theorem in the case where the flow $\alpha$ is \emph{prime}, i.e.\! when the fixed point algebra $R^\alpha$ is a factor. We exploit the structure of prime almost periodic flows on the hyperfinite $\II_1$ factor \cite{Jo82, OPT79} to show by hand that $\alpha : \R \curvearrowright R$ has the Rokhlin property. Our argument is inspired by \cite[Theorem 6.12]{MT12} and \cite[Proposition 2.5]{Ki95}.

The second step, which is the key novelty of our paper, is a reduction from the general case to the prime case by a cocycle perturbation argument. In fact, we prove that every almost periodic flow $\alpha$ such that $\Gamma(\alpha)=\R$ is cocycle conjugate to an almost periodic flow that is prime.
 
\begin{letterthm} \label{cocycle perturbation flow}
    Let $M$ be a factor with separable predual and $\alpha : \R \curvearrowright M$ an almost periodic flow such that $\Gamma(\alpha) = \R$. Then $\alpha$ is cocycle conjugate to an almost periodic flow $\beta : \R \curvearrowright M$ such that  $M^\alpha \subset M^\beta$ and $M^\beta$ is a factor.

    More precisely, we have $\beta_t= \Ad(v_t) \circ \alpha_t$ for some continuous homomorphism $v : \R \rightarrow \cU(\cZ(M^\alpha))$.
\end{letterthm}

Quite unexpectedly, in order to prove Theorem \ref{cocycle perturbation flow}, one has to deal with type $\III$ equivalence relations, even if one is only interested in the case where $M$ is a $\II_1$ factor. Indeed, the almost periodicity of the flow $\alpha : \R \curvearrowright M$ implies that $M^\alpha$ is stably normalized by the eigenvectors of $\alpha$ and this induces an amenable equivalence relation $\cR$ on the center $\cZ(M^\alpha)$. The morphism $v : \R \rightarrow \cU(\cZ(M^\alpha))$ of Theorem \ref{cocycle perturbation flow}, is obtained by studying the cohomology of $\cR$ and it turns out that this equivalence relation can be of arbitrary type (possibly type $\III$) even when $M$ is a $\II_1$ factor. This phenomenon was already observed in \cite{BHV15}.

\subsection{Cohomology of amenable ergodic equivalence relations}

The proof of Theorem \ref{cocycle perturbation flow} relies on a key result regarding the cohomology of amenable ergodic nonsingular equivalence relations, which is of independent interest. In order to state this result, we introduce some further terminology.

Let $\mathcal R$ be an ergodic nonsingular equivalence relation defined on a diffuse standard probability space $(X, \nu)$. Let $G$ be a second countable locally compact group with a left invariant Haar measure $m_G$ and $c : \mathcal R \to G$ a measurable cocycle. Following \cite{GS91}, we say that the cocycle $c$ is {\em ergodic} if for every nonnull measurable subset $Y \subset X$, the essential image of $c|_{\mathcal R_Y} : \mathcal R_Y \to G$ is equal to $G$. Here, $\mathcal R_Y = \mathcal R \cap (Y \times Y)$ denotes the restricted ergodic nonsingular equivalence relation on the diffuse standard probability space $(Y, \nu_Y)$, where $\nu_Y = \frac{1}{\nu(Y)}\nu|_Y$. Clearly, if $\ker(c) < \cR$ is an ergodic subequivalence relation and the essential image of $c$ is equal to $G$, then $c$ is ergodic, but the converse is not true (except when $G$ is discrete).

In \cite{GS91}, Golodets--Sinelshchikov proved on the one hand, that if a cocycle $c : \mathcal R \to G$ is ergodic, then it is always \emph{cohomologous} to a measurable cocycle $d : \mathcal R \to G$ for which $\ker(d) < \mathcal R$ is ergodic. On the other hand, they proved that if $\mathcal R$ is amenable (or hyperfinite by \cite{CFW81}), then for any countable dense subgroup $\Lambda < G$, the cocycle $c$ is cohomologous to a measurable cocycle $d : \mathcal R \to G$ with values in $\Lambda$. 

Now, assume at the same time that $\mathcal R$ is amenable and that $c : \mathcal R \to G$ is ergodic (under these conditions, $G$ must be amenable). Can we then find a measurable cocycle $d : \mathcal R \to G$ cohomologous to $c$ such that at the same time $\ker(d) < \mathcal R$ is ergodic and $d$ takes values in a dense countable subgroup? If such a cocycle $d$ exists, then it must take values in an \emph{amenable} dense countable  subgroup of $G$. Our key result shows that this necessary condition is actually sufficient.

\begin{letterthm} \label{ergodic kernel dense subgroup}
    Let $\cR$ be an ergodic equivalence relation and  $c : \mathcal R \to G$ a measurable cocycle with values in a second countable locally compact group $G$. Suppose that $\cR$ is amenable, that $c$ is ergodic and that $G$ contains a dense countable amenable subgroup $\Lambda < G$.
    
    Then $c$ is cohomologous to a measurable cocycle $d : \cR \rightarrow G$ with values in $\Lambda$ and such that the subequivalence relation $\ker(d) < \cR$ is ergodic.
\end{letterthm}

The proof of Theorem \ref{ergodic kernel dense subgroup} relies on Golodets--Sinelshchikov's classification theorem for measurable cocycles on amenable ergodic nonsingular equivalence relations (see \cite[Theorem 3.1]{GS91}) and a novel construction of model actions using the notion of adjoint flows of \cite{VV22}.

\subsection{Extremal almost periodic weights}

Let $M$ be a factor with separable predual and denote by $\mathcal P(M)$ the set of all faithful normal semifinite weights on $M$. Following \cite{Co72,Co74}, we say that $\varphi \in \mathcal P(M)$ is {\em almost periodic} if  the modular automorphism group $\sigma^\varphi : \R \curvearrowright M$ is almost periodic. Almost periodic weights are of the utmost importance in the classification theory of type $\III$ factors. Indeed, Connes \cite{Co74} showed that when $M$ is a type $\III$ factor, to any almost periodic weight $\varphi \in \mathcal P(M)$, one can associate a canonical discrete decomposition of the form $M = N \rtimes_\theta \Gamma$, where $N$ is a type ${\rm II_\infty}$ von Neumann algebra and $\theta : \Gamma \curvearrowright N$ is a trace-scaling action for some countable subgroup $\Gamma < \R^*_+$. Using the above discrete decomposition, the classification problem for type $\III$ factors reduces to the classification of type ${\rm II_\infty}$ von Neumann algebras and their outer automorphisms. In \cite{Co72, Co75a, Co75b}, Connes applied this approach with great success to the structure and the classification of factors of type ${\rm III_\lambda}$ with $\lambda \in [0, 1)$.

Let $M$ be a type $\III$ factor with separable predual and $\varphi \in \mathcal P(M)$ an almost periodic weight. We then simply say that $M$ is an almost periodic type $\III$ factor. We  say that $\varphi \in \mathcal P(M)$ is an \emph{extremal} almost periodic weight if the centralizer von Neumann algebra $M_\varphi$ is a factor. Any extremal almost periodic weight $\varphi \in \mathcal P(M)$ gives rise to a {\em factorial} discrete decomposition $M = N \rtimes_\theta \Gamma$, where $N$ is a type ${\rm II_\infty}$ factor. While every type $\III_\lambda$ factor for $\lambda \in (0,1)$ has an extremal periodic weight, namely the $\frac{2\pi}{|\log(\lambda)|}$-periodic weight, it is known that a type $\III_0$ factor has no extremal almost periodic weight \cite{Co72}. In \cite{Co74}, Connes proved that  any \emph{full} almost periodic factor of type $\III_1$ with separable predual has an extremal almost periodic weight. In his argument, the fullness assumption is essential. As a byproduct of Theorem \ref{cocycle perturbation flow}, we are able to remove the fullness assumption.

\begin{letterthm} \label{main extremal}
Let $M$ be an almost periodic factor of type $\III_1$ with separable predual. Then $M$ has an extremal almost periodic weight $\varphi \in \mathcal P(M)$. In particular, $M$ has a factorial discrete decomposition.
\end{letterthm}
 
We use Theorem \ref{main extremal} and its more precise version (see Theorem \ref{precise extremal}) to obtain a new  characterization of extremal almost periodic weights (see Corollary \ref{cor extremal} below).

\subsection*{Acknowledgments} We are grateful to Sorin Popa for his insightful comments.


\section{Preliminaries}\label{preliminaries}

\subsection{Cocycle conjugacy}

Let $M$ be a von Neumann algebra with separable predual, $G$ a second countable locally compact group and $\alpha : G \curvearrowright M$ a continuous action. We denote by $M \rtimes_\alpha G$ the crossed product von Neumann algebra generated by $M$ and a copy of the left regular representation $\lambda_\alpha : G \to \mathcal U(M \rtimes_\alpha G) : g \mapsto \lambda_\alpha(g)$ in such a way that 
$$\forall g \in G, \forall x \in M, \quad \alpha_g(x) = \lambda_\alpha(g) x \lambda_\alpha(g)^*.$$
A strongly continuous map $u : G \to \mathcal U(M) : g \mapsto u_g$ is said to be a $1$-{\em cocycle} for $\alpha$ if $u_{gh} = u_g \alpha_g(u_h)$ for all $g, h \in G$. We denote by $\rZ^1(\alpha, G, M)$ the space of all $1$-cocycles for $\alpha$. If $u \in \rZ^1(\alpha, G, M)$, then we may define a new continuous action $\alpha^u : G \curvearrowright M$ by the formula $\alpha_g^u = \Ad(u_g) \circ \alpha_g$ for every $g \in G$.

\begin{definition}\label{cocycle conjugacy}
Let $\alpha, \beta : G \curvearrowright M$ be continuous actions. We say that $\alpha$ and $\beta$ are {\em cocycle conjugate} if there exist an automorphism $\theta \in \Aut(M)$ and a $1$-cocycle $u \in \rZ^1(\alpha, G, M)$ such that 
$$\forall g \in G, \quad \alpha^u_g = \theta \circ \beta_g \circ \theta^{-1}.$$
If we can take $\theta=\id_M$, we say that $\alpha$ and $\beta$ are {\em cocycle equivalent}.
\end{definition}

If $\alpha$ and $\beta$ are cocycle conjugate with $\theta \in \Aut(M)$ and $u \in \rZ^1(\alpha, G, M)$ as in Definition \ref{cocycle conjugacy}, then the mapping 
$$\pi_{\alpha, \beta} : M \rtimes_\beta G \to M \rtimes_\alpha G : x \lambda_\beta(g) \mapsto \theta(x) u_g \lambda_\alpha(g)$$
extends to a well-defined isomorphism such that $\pi_{\alpha, \beta}(M) = M$. In other words, the inclusions $M \subset M \rtimes_\alpha G$ and $M \subset M \rtimes_\beta G$ are isomorphic.

\subsection{The Rokhlin property}

Let $M$ be a von Neumann algebra with separable predual. We denote by $\ell^\infty(\N, M)$ the unital $\rC^*$-algebra of all norm bounded sequences in $M$. Fix a nonprincipal ultrafilter $\omega \in \beta(\N)\setminus \N$. Define the $\rC^*$-subalgebras $\mathcal I_\omega, \mathcal M^\omega \subset \ell^\infty(\N, M)$ by the formulae
\begin{align*}
	\mathcal I_{\omega} &= \left\{ x = (x_n)_{n} \in \ell^\infty(\N, M) \mid x_n \to 0 \text{ $\ast$-strongly as } n \to \omega \right\} \\
	\mathcal M^{\omega} &= \left \{ x = (x_n)_n \in \ell^\infty(\N, M) \mid  x \mathcal I_{\omega} \subset \mathcal I_{\omega} \text{ and } \mathcal I_{\omega}x \subset \mathcal I_{\omega}\right\}.
\end{align*}
The quotient $\rC^*$-algebra $M^\omega = \mathcal{M}^\omega/ \mathcal{I}_\omega$ is a von Neumann algebra, and we simply call it the \textit{ultraproduct von Neumann algebra} \cite{Oc85}. For any faithful normal state $\varphi\in M_*$, the assignment $\mathcal M^\omega \to \C : (x_n)_n \mapsto \lim_{n\to \omega} \varphi(x_n)$ induces a faithful normal state on $M^\omega$, which we write $\varphi^\omega$. 
There is a natural normal embedding $M \subset M^\omega$ with faithful normal expectation $\rE_M : M^\omega \to M$ which satisfies $\varphi^\omega = \varphi \circ \rE_M$. For any $(x_n)_n\in \mathcal M^\omega$, we denote by $(x_n)^\omega$ its image in $M^\omega$. We also consider the {\em asymptotic centralizer} $M_\omega = (M' \cap M^\omega)_{\varphi^{\omega}}$, which is the centralizer with respect to the ultraproduct state $\varphi^\omega$ of the central sequence algebra $M' \cap M^\omega$ \cite{Co74, AH12}. We observe that the asymptotic centralizer $M_\omega$ does not depend on the choice of the faithful normal state $\varphi \in M_\ast$.

Let $G$ be a second countable locally compact group and $\alpha : G \curvearrowright M$ a continuous action. The ultraproduct action $\alpha^\omega : G \curvearrowright M^\omega$ is quite discontinuous in general. Following \cite[Section 3]{MT12}, we introduce the von Neumann subalgebra $M^\omega_\alpha \subset M^\omega$ of all $(\alpha, \omega)$-equicontinuous elements on which the restriction of the ultraproduct action $\alpha^\omega : G \curvearrowright M^\omega$ becomes continuous. For this, let us fix a faithful normal state $\varphi \in M_\ast$. A sequence $(x_n)_n \in \mathcal M^\omega$ is said to be $(\alpha, \omega)$-{\em equicontinuous} if for every $\varepsilon > 0$, there exists an open neighborhood $U \subset G$ of the neutral element $e \in G$ such that 
$$\left\{n \in \N \mid \forall g \in U, \|\alpha_g(x_n) - x_n\|_\varphi^\sharp < \varepsilon \right\} \in \omega.$$
Denote by $\mathcal E_\alpha^\omega \subset \mathcal M^\omega$ the unital $\rC^*$-subalgebra of all $(\alpha, \omega)$-equicontinuous elements. We observe that $\mathcal E_\alpha^\omega$ does not depend on the choice of the faithful normal state $\varphi \in M_\ast$. Then we have $\mathcal I_\omega \subset \mathcal E_\alpha^\omega \subset \mathcal M^\omega$ and the quotient $\rC^*$-algebra $M_\alpha^\omega = \mathcal E_\alpha^\omega/ \mathcal{I}_\omega$ is a von Neumann subalgebra of $M^\omega$. Moreover $M_\alpha^\omega \subset M^\omega$ is globally invariant under the ultraproduct action $\alpha^\omega$ and the restriction $\alpha^\omega : G \curvearrowright M_\alpha^\omega$ defines a continuous action. We also consider the $(\alpha, \omega)$-equicontinuous part $M_{\alpha, \omega} = M_\alpha^\omega \cap M_\omega$ of the asymptotic centralizer and the corresponding continuous action $\alpha^\omega : G \curvearrowright M_{\alpha, \omega}$.

Following \cite[Section 4]{MT12}, we introduce the Rokhlin property for continuous actions of arbitrary abelian second countable locally compact groups on von Neumann algebras.

\begin{definition}\label{Rokhlin property}
Let $G$ be an abelian second countable locally compact group and $\alpha : G \curvearrowright M$ a continuous action. We say $\alpha : G \curvearrowright M$ satisfies the {\em Rokhlin property} if for every $p \in \widehat G$, there exists a unitary $u_p \in \mathcal U(M_{\alpha, \omega})$ such that $\alpha^\omega_g(u_p) = \langle g, p\rangle u_p$ for every $g \in G$.
\end{definition}

We observe that the Rokhlin property in Definition \ref{Rokhlin property} does not depend on the choice on the nonprincipal ultrafilter $\omega \in \beta(\N) \setminus \N$.

\section{The Maharam extension and the adjoint flows}

\subsection{The Maharam extension}
Let $M$ be a von Neumann algebra. Then, up to a unique isomorphism, there exists a unique triple $(\widetilde{M},\tau,\theta)$ where $\widetilde{M}$ is a von Neumann algebra containing $M$, $\tau$ is a faithful semifinite trace on $\widetilde{M}$ and $\theta : \R^*_+ \curvearrowright \widetilde{M}$ is an action that scales the trace $\tau$, that is $\tau \circ \theta_\lambda =\lambda \tau$ for all $\lambda \in \R^*_+$. We call $(\widetilde{M},\tau,\theta)$ the \emph{noncommutative flow of weights} of $M$ (see \cite[Chapter ${\rm XII}$]{Ta03}). 

Suppose that $\alpha : G \curvearrowright M$ is a continuous action of a locally compact group $G$. Then $\alpha$ admits a unique extension $\tilde{\alpha} : G \curvearrowright \widetilde{M}$ that preserves the trace $\tau$ and commutes with the scaling action $\theta$. We call $\tilde{\alpha}$ the \emph{Maharam extension} of $\alpha$.

\begin{example}
    Suppose that $\alpha : G \curvearrowright (X,\nu)$ is a nonsingular action on a $\sigma$-finite measure space. We can view $\alpha$ as an action on an abelian von Neumann algebra $M$. The explicit formula for its Maharam extension is given as follows. Letting $M=\rL^\infty(X,\nu)$, we have $\widetilde{M}=\rL^\infty(X \times \R^*_+,\nu \otimes m)$ with trace $\tau=\nu \otimes m$ where $m$ is the measure on $\R^*_+$ defined by $\rd m(\lambda) = \lambda^{- 2}\rd \lambda$, and the trace scaling action $\theta$ is given by $\theta_\lambda(x,s)=(x,\lambda^{-1}s)$. The Maharam extension $\widetilde{\alpha} : G \curvearrowright (X \times \R^*_+, \nu \otimes m)$ is given by
    $$ \widetilde{\alpha}_g(x,s)=(\alpha_g(x), \omega(g,x) s), \quad g \in G, \; (x,s) \in X \times \R^*_+$$
    where $\omega(g,x)=\frac{\rd \nu (\alpha_g(x))}{\rd \nu(x)}$ is the \emph{Radon--Nikodym cocycle}. 
\end{example}

\subsection{Adjoint flows}
Let $\alpha : \R^*_+ \curvearrowright M$ be a continuous flow where $M$ is a von Neumann algebra. Let $(\widetilde{M},\tau,\theta)$ be the noncommutative flow of weights of $M$ where $\theta : \R^*_+ \curvearrowright \widetilde{M}$ scales the trace $\tau$. Let $\widetilde{\alpha} : \R^*_+ \curvearrowright \widetilde{M}$ be the natural extension of $\alpha$ to $\widetilde{M}$. Recall that $\widetilde{\alpha}$ commutes with $\theta$. For $t \in \R$, we define 
$$ M^{(t, \alpha)}=\left\{ x \in \widetilde{M} \mid \forall \lambda \in \R^*_+, \; \widetilde{\alpha}_{\lambda^t}(x )=\theta_\lambda(x) \right\}.$$

We denote by $\alpha^{(t)} : \R^*_+ \curvearrowright M^{(t,\alpha)}$ the restriction of $\widetilde{\alpha}$ to $M^{(t,\alpha)}$. Observe that $M^{(0,\alpha)}=M$ and $\alpha^{(0)}=\alpha$.

Take $t,s \in \R$. We claim that for any flow $\alpha : \R^*_+ \curvearrowright M$, the flow $(\alpha^{(t)})^{(s)}$ is naturally isomorphic to $\alpha^{(t+s)}$.

Indeed, let $N=M^{(t,\alpha)}$ and $\beta = \alpha^{(t)}$. Then, by definition, $N$ is the fixed point algebra of $\widetilde{M}$ under the trace scaling action $\theta' : \lambda \mapsto \widetilde{\alpha}_{\lambda^{-t}}\circ \theta_\lambda$. Thus $(\widetilde{M},\tau,\theta')$ is naturally identified with the noncommutative flow of weights of $N$. Since $\widetilde{\alpha}|_N = \beta$ and $\widetilde{\alpha}$ preserves $\tau$ and commutes with $\theta'$, we then see that the natural extension $\widetilde{\beta}$ of $\beta$ to the noncommutative flow of weights of $N$ is identified with $\widetilde{\alpha}$. Therefore, $N^{(s,\beta)}$ is the fixed point algebra of $\widetilde{M}$ under the action 
$$\lambda \mapsto \widetilde{\beta}_{\lambda^{-s}} \circ \theta'_\lambda=\widetilde{\alpha}_{\lambda^{-s}} \circ \theta'_\lambda = \widetilde{\alpha}_{\lambda^{-s}} \circ \widetilde{\alpha}_{\lambda^{-t}}\circ \theta_\lambda = \widetilde{\alpha}_{\lambda^{-(t+s)}}\circ \theta_\lambda.$$
This shows that $N^{(s,\beta)}=M^{(t+s,\alpha)}$ and $\beta^{(s)}$ is the restriction of $\widetilde{\beta}=\widetilde{\alpha}$ to $N^{(s,\beta)}=M^{(t+s,\alpha)}$, which is precisely $\alpha^{(t+s)}$.

\begin{example}
    Suppose that $\alpha : \R^*_+ \curvearrowright (X,\nu)$ is a nonsingular flow on a $\sigma$-finite measure space. We can view $\alpha$ as an action on an abelian von Neumann algebra and construct a new nonsingular flow $\alpha^{(t)}$. The explicit formula for this flow is given as follows. Letting $M=\rL^\infty(X,\nu)$, we have $\widetilde{M}=\rL^\infty(X \times \R^*_+,\nu \otimes m)$ with trace $\tau=\nu \otimes m$ where $m$ is the measure on $\R^*_+$ defined by $\rd m(\lambda) = \lambda^{- 2}\rd \lambda$, and the trace scaling action $\theta$ is given by $\theta_\lambda(x,s)=(x,\lambda^{-1}s)$. The Maharam extension $\widetilde{\alpha} : \R^*_+ \curvearrowright (X \times \R^*_+, \nu \otimes m)$ is given by
    $$ \widetilde{\alpha}_\lambda(x,s)=(\alpha_\lambda(x), \omega(\lambda,x) s)$$
    where $\omega(\lambda,x)=\frac{\rd \nu (\alpha_\lambda(x))}{\rd \nu(x)}$ is the Radon--Nikodym cocycle. Then for $t \in \R$, the algebra $M^{(t,\alpha)}$ consists of all functions $F \in \rL^\infty(X \times \R^*_+)$ such that $$F(\alpha_{\lambda^t}(x), \lambda \omega(\lambda^t,x) s) =F(x,s)$$ for all $\lambda \in \R^*_+$ and all $(x,s) \in X \times \R^*_+$. 
\end{example}

\begin{remark}
Let $\alpha : \R^*_+ \curvearrowright M$ be a flow. We denote by $\alpha^{\op}$ the flow given by $(\alpha^{\op})_\lambda= \alpha_{\lambda^{-1}}$ for all $\lambda \in \R^*_+$. Then it is easy to see that $M^{(t,\alpha^{\op})} = M^{(-t,\alpha)}$ and $(\alpha^{\op} )^{(t)}= (\alpha^{(-t)})^{\op}$. Define $\alpha^{\dagger}=(\alpha^{(1)})^{\op}=(\alpha^{\op} )^{(-1)}$. Then it follows from the foregoing that $(\alpha^{\dagger})^{\dagger}$ is naturally isomorphic to $\alpha$. The flow $\alpha^\dagger$ coincides with the {\em adjoint flow} of $\alpha$ defined in \cite{VV22} in the case where $M$ is an abelian von Neumann algebra. 

When $\alpha$ is trace preserving, it is easy to see that $\alpha^{(t)} \cong \alpha$ for every $t \in \R$. In particular, we have $\alpha^\dagger \cong \alpha^{\op}$. Note that in \cite[Section 3]{VV22}, it is incorrectly stated that if $\alpha$ is trace preserving, then $\alpha^\dagger \cong \alpha$.
\end{remark}

The construction $(M,\alpha) \mapsto (M^{(t,\alpha)},\alpha^{(t)})$ is functorial with respect to $\R^*_+$-equivariant isomorphisms. More precisely, if $\beta : \R^*_+ \curvearrowright N$ is another flow and $\pi : M \rightarrow N$ is an isomorphism of von Neumann algebras that conjugates $\alpha$ with $\beta$, then $\pi$ extends to an isomorphism $\widetilde{\pi} : \widetilde{M} \rightarrow \widetilde{N}$ that conjugates $\widetilde{\alpha}$ with $\widetilde{\beta}$ and $\theta_M$ with $\theta_N$. In particular, $\widetilde{\pi}$ restricts to an isomorphism from $M^{(t,\alpha)}$ onto $N^{(t,\beta)}$ that conjugates $\alpha^{(t)}$ with $\beta^{(t)}$. 

Set $\Aut(M, \alpha) =\left\{\pi \in \Aut(M) \mid \pi \text{ commutes with } \alpha\right\}$. By taking $M=N$, we see that the map 
$$\Aut(M,\alpha)\to \Aut(M^{(t,\alpha)},\alpha^{(t)}) : \pi \mapsto \pi^{(t)}$$ 
induces a continuous homomorphism.
In particular, if $G$ is a locally compact group and $\rho : G \curvearrowright M$ is a continuous action that commutes with a flow $\alpha : \R^*_+ \curvearrowright M$, then we obtain a continuous action $\rho^{(t)} : G\curvearrowright M^{(t,\alpha)}$ that commutes with $\alpha^{(t)}$. Again, we have that the action $(\rho^{(t)})^{(s)}$ is naturally isomorphic to $\rho^{(t+s)}$ for every $s,t \in \R$.


\section{Induced equivalence relations}
Let $\cR$ be a nonsingular equivalence relation on a measure space $(X,\nu)$ and let $c : \cR \rightarrow G$ be a cocycle with values in a second countable locally compact group $G$. Let $\rho : G \curvearrowright (Y,\eta)$ be a nonsingular action. The \emph{induced equivalence relation} $\cR_{\rho,c}$ is the nonsingular equivalence relation on $(X \times Y, \nu \otimes \eta)$ defined by
$$ (x,y) \sim_{\cR_{\rho, c}} (x',y') \Longleftrightarrow x \sim_{\cR} x' \text{ and } y' = \rho_{c(x',x)}(y).$$
When $\rho : G \curvearrowright G$ is the left translation action, the equivalence relation $\cR_{\rho,c}$ is simply denoted $\cR_{c}$ and it is called the \emph{skew product equivalence relation}.

\begin{lemma} \label{Radon--Nikodym induced equivalence relation}
    Let $\cR$ be a nonsingular equivalence relation on $(X,\nu)$ with Radon--Nikodym cocycle $\delta : \cR \rightarrow \R^*_+$. Let $c : \cR \rightarrow G$ be a cocycle with values in the second countable locally compact group $G$. Let $\rho : G \curvearrowright (Y,\eta)$ be a nonsingular action with Radon--Nikodym cocycle $\sigma : G \times Y \rightarrow \R^*_+$. Then the Radon--Nikodym cocycle $\omega : \cR_{\rho,c} \rightarrow \R^*_+$ of $\cR_{\rho,c}$ with respect to $\nu \otimes \eta$ is given by
    $$ \omega(x,y,x',y')=\delta(x,x') \sigma(c(x,x'),y')$$
    for all $(x,y,x',y') \in \cR_{\rho,c}$.
\end{lemma}
\begin{proof}
    The Radon--Nikodym cocycle $\delta$ is characterized by
    $$ \int_X \sum_{x \sim_{\cR} x'} f(x,x') \delta(x,x') \: \rd \nu(x') = \int_X \sum_{x' \sim_{\cR} x} f(x,x') \: \rd \nu(x)$$
    for every positive measurable function $f$ on $X \times X$.

    For every positive measurable function $f$ on $X \times Y \times X \times Y$, we have 
    $$ \int_{X \times Y} \sum_{(x,y) \sim_{\cR_{\rho,c}} (x',y')} f(x,y,x',y') \omega(x,y,x',y') \: \rd \nu(x') \rd \eta(y') =$$
    $$\int_{X\times Y} \sum_{(x',y') \sim_{\cR_{\rho,c}} (x,y)} f(x,y,x',y') \: \rd \nu(x) \rd \eta(y)=$$
    $$\int_{X\times Y} \sum_{x' \sim_\cR x} f(x,y,x',\rho_{c(x',x)}(y)) \: \rd \nu(x) \rd \eta(y)=$$

    $$\int_{X\times Y} \sum_{x' \sim_\cR x} f(x,\rho_{c(x,x')}(y'),x',y') \sigma(c(x,x'),y') \: \rd \nu(x) \rd \eta(y')=$$
    $$\int_{X\times Y} \sum_{x \sim_\cR x'} f(x,\rho_{c(x,x')}(y'),x',y') \sigma(c(x,x'),y') \delta(x,x') \: \rd \nu(x') \rd \eta(y')=$$
    $$\int_{X\times Y} \sum_{(x,y) \sim_{\cR_{\rho,c}} (x',y')} f(x,y,x',y') \sigma(c(x,x'),y') \delta(x,x') \: \rd \nu(x') \rd \eta(y').$$
    This shows the desired equality.
\end{proof}

Let $\cR$ be an ergodic nonsingular equivalence relation on $(X,\nu)$. Let $c : \cR \rightarrow G$ be a cocycle with values in a second countable locally compact group $G$. The right translation action $\theta : G \curvearrowright X \times G$ given by $\theta_g(x,h)=(x,hg^{-1})$ preserves the skew product equivalence relation $\cR_{c}$. Therefore, $\theta$ induces a $G$-action on the space of ergodic components $(X \times G)/{\cR_c}$. This $G$-action is called the \emph{Mackey range} of $c$.

On the functional level, the Mackey range is simply the restriction of the right translation action $\theta : G \curvearrowright \rL^\infty(X \times G)$ to the algebra of $\cR_c$-invariant functions $\rL^\infty(X \times G)^{\cR_c}$. 

\begin{proposition} \label{ergodic skew product}
    Let $\cR$ be an ergodic nonsingular equivalence relation on $(X,\nu)$. Let $c : \cR \rightarrow G$ be a cocycle with values in a second countable locally compact group $G$. The cocycle $c$ is ergodic if and only if its Mackey range is trivial or equivalently if and only if $\cR_c$ is ergodic.
\end{proposition}

\begin{example}
Let $\cR$ be a nonsingular equivalence relation on $(X,\nu)$ with its Radon--Nikodym cocycle $\delta : \cR \rightarrow \R^*_+$. Then we see from Lemma \ref{Radon--Nikodym induced equivalence relation} that the equivalence relation $\cR_{\delta}$ preserves the measure $\nu \otimes m$ on $X \times \R^*_+$, where $m$ is the measure on $\R^*_+$ defined by $\rd m(\lambda) = \lambda^{- 2}\rd \lambda$. The equivalence relation $\cR_\delta$ is called the \emph{Maharam extension} of $\cR$.

The action $\theta : \R^*_+ \curvearrowright X \times \R^*_+$ given by $\theta_\lambda(x,s)=(x,s\lambda^{-1})$ scales the measure $\nu \otimes m$ in the sense that $(\nu \otimes m)(\theta_\lambda(A))=\lambda (\nu \otimes m)(A)$ for all $A \subset X \times \R^*_+$ and all $\lambda \in \R^*_+$.

The Mackey range of $\delta$ is called the \emph{Krieger flow} of $\cR$. It is the $\R^*_+$-action induced by $\theta$ on the ergodic components of the Maharam extension $\cR_\delta$.

We say that an ergodic equivalence relation $\cR$ is of type $\III_1$ if its Radon--Nikodym cocycle $\delta$ is ergodic or equivalently if its Maharam extension $\cR_\delta$ is ergodic, or in other words, if the Krieger flow of $\cR$ is trivial.
\end{example}
We end this section with the following useful ergodicity lemma.

\begin{lemma} \label{dense range induced equivalence relation}
    Let $\cR$ be an ergodic nonsingular equivalence relation on $(X,\nu)$. Let $c : \cR \rightarrow G$ be a cocycle with values in a second countable locally compact group $G$. Let $\rho : G \curvearrowright (Y,\eta)$ be a nonsingular action. If $c$ is ergodic, then $\rL^\infty(X \times Y, \nu \otimes \eta)^{\cR_{\rho,c}} = \rL^\infty(Y)^{\rho}$.
\end{lemma}
\begin{proof}
    Let $\varphi : X \times G \times Y \rightarrow X \times Y$ be the map defined by $\varphi(x,g,y)=(x,\rho_g(y))$. Take $f \in \rL^\infty(X \times Y, \nu \otimes \eta)^{\cR_{\rho,c}}$. Then $$f \circ \varphi \in \rL^\infty(X \times G \times Y, \nu \otimes m_G \otimes \eta)^{\cR_{c} \times \Delta_Y}$$
    where $\cR_c$ is the skew product equivalence relation on $X \times G$ and $\Delta_Y$ is the diagonal (or trivial) equivalence relation on $Y$. Since $c$ is ergodic, then $\cR_c$ is ergodic by Proposition \ref{ergodic skew product}, hence $f \circ \varphi$ depends only on the third variable. We conclude that $f$ depends only on the second variable and is $\rho$-invariant.
\end{proof}


\section{Cohomology of amenable ergodic equivalence relations}

Our goal in this section is to prove Theorem \ref{ergodic kernel dense subgroup}. Observe that Theorem \ref{ergodic kernel dense subgroup} is somehow a combination of the following two lemmas.

\begin{lemma}[{\cite[Lemma 1.6]{GS91}}] \label{ergodic kernel}
Let $\cR$ be an ergodic nonsingular equivalence relation and let $c : \cR \rightarrow G$ be a cocycle with values in a second countable locally compact group $G$. 

Assume that $c$ is ergodic. Then $c$ is cohomologous to a cocycle $d : \cR \rightarrow G$ such that the subequivalence relation $\ker(d) < \cR$ is ergodic.
\end{lemma}

\begin{lemma}[{\cite[Proposition 1.2]{GS91}}] \label{dense subgroup}
Let $\cR$ be an ergodic nonsingular equivalence relation and let $c : \cR \rightarrow G$ be a cocycle with values in a second countable locally compact group $G$. 

Assume that $\cR$ is amenable and let $\Lambda < G$ be a dense countable subgroup. Then $c$ is cohomologous to a cocycle $d : \cR \rightarrow G$ with values in $\Lambda$.
\end{lemma}

Unfortunately, we do not know how to prove Theorem \ref{ergodic kernel dense subgroup} by combining the inductive constructions of Lemma \ref{ergodic kernel} and Lemma \ref{dense subgroup} in a single construction. In fact, we observe that the assumption that $\Lambda$ is amenable is essential in Theorem \ref{ergodic kernel dense subgroup} while it is not in Lemma \ref{dense subgroup}.

Instead of a direct proof based on an inductive construction, we will prove Theorem \ref{ergodic kernel dense subgroup} by using the full power of the classification of cocycles on amenable equivalence relations obtained in \cite{GS91}. Let us recall this classification theorem.

Let $\cR$ be a nonsingular equivalence relation on $(X,\nu)$, $G$ a second countable locally compact group and $c, d :  \cR \rightarrow G$ two measurable cocycles. 
We say that $c$ and $d$ are {\em weakly equivalent} if there exists $\theta \in\Aut(\mathcal R)$ such that $c$ and $d \circ (\theta \times \theta)$ are cohomologous.

\begin{theorem}[{\cite[Theorem 3.1]{GS91}}] \label{classification of cocycles}
Let $\cR$ be an ergodic amenable nonsingular equivalence relation on a standard probability space $(X, \nu)$. Let $\omega : \cR \rightarrow \R^*_+$ be the Radon--Nikodym cocycle associated to $\nu$. Let $c, d : \cR \rightarrow G$ be two ergodic cocycles with values in a second countable locally compact group $G$. 

Then $c$ and $d$ are weakly equivalent if and only if the Mackey ranges of the cocycles $c \times \omega, d \times \omega : \cR \rightarrow G \times \R^*_+$ are conjugate.
\end{theorem}

Thanks to this classification theorem, we can prove Theorem \ref{ergodic kernel dense subgroup} by building model cocycles with prescribed Mackey range that satisfy the desired properties. More precisely we will prove the following.

\begin{theorem} \label{model ergodic kernel dense subgroup}
    Let $G$ be a second countable locally compact group. Let $\rho : G \times \R^*_+ \curvearrowright (Y,\eta)$ be any ergodic nonsingular action. Then there exists an ergodic nonsingular equivalence relation $\cR$ on a probability space $(X,\nu)$ and a cocycle $c : \cR \rightarrow G$ such that the Mackey range of $c \times \omega : \cR \rightarrow G \times \R^*_+$ is isomorphic to $\rho$, where $\omega : \cR \rightarrow \R^*_+$ is the Radon--Nikodym cocycle.

    Moreover, the following holds :
    \begin{enumerate}
        \item If $G$ is amenable, we can take $\cR$ to be amenable.
        \item If $\rho|_{\R^*_+}$ is ergodic, we can take $c$ such that $\ker(c)$ is ergodic.
        \item If $\Lambda < G$ is a dense countable subgroup, we can take $c$ with values in $\Lambda$.
    \end{enumerate}
    Any two of these conditions can be satisfied simultaneously. If $\Lambda$ is amenable, the three of them can be satisfied simultaneously.
\end{theorem}

We first observe that Theorem \ref{model ergodic kernel dense subgroup} holds in the case where $\rho : G \times \R^*_+ \curvearrowright \R^*_+$ is the action given by $\rho_{(g,\lambda)}(x)=\lambda x$ for all $(g,\lambda,x) \in G \times \R^*_+ \times \R^*_+$. This means that $c$ is ergodic and $\omega$ is a coboundary ($\cR$ is measure preserving).

\begin{theorem} \label{model ergodic kernel dense subgroup measure preserving case}
    Let $G$ be a second countable locally compact group. Then there exists an ergodic probability measure preserving equivalence relation $\cR$ with an ergodic cocycle $c : \cR \rightarrow G$.

    Moreover, the following holds :
    \begin{enumerate}
        \item If $G$ is amenable, we can take $\cR$ to be amenable.
        \item We can take $c$ such that $\ker(c)$ is ergodic.
        \item If $\Lambda < G$ is a dense countable subgroup, we can take $c$ with values in $\Lambda$.
    \end{enumerate}
    Any two of these conditions can be satisfied simultaneously. If $\Lambda$ is amenable, the three of them can be satisfied simultaneously.
\end{theorem} 
\begin{proof}
    If $G$ is amenable, it is known that there exists an ergodic cocycle $c : \cR \rightarrow G$ where $\cR$ is a hyperfinite ergodic probability measure preserving equivalence relation \cite{GS83a, GS83b}. Thanks to Lemma \ref{ergodic kernel}, up to replacing $c$ with a cohomologous cocycle, we can ensure that $\ker(c)$ is ergodic. If we use instead Lemma \ref{dense subgroup}, we can ensure that $c$ takes its values in $\Lambda$.

    If $G$ is not necessarily amenable, let $F$ be a nontrivial finite group and consider the wreath product $H = \Lambda \ltimes F^{(\Lambda)}$ with its natural free and ergodic probability measure preserving action on the compact group $(X, \nu) = F^\Lambda$. Let $\cR$ be the associated orbit equivalence relation.  Denote by  $p : H \to \Lambda$ the canonical quotient homomorphism. Let $c : \cR \rightarrow G$ be the cocycle defined by $c(hx, x)= p(h)$ for all $h \in H$ and $x \in X$. Then $\ker(c)$ is ergodic since it contains the orbit equivalence relation of $F^{(\Lambda)}$ and $c : \cR \rightarrow G$ is ergodic since $\Lambda < G$ is dense. Clearly, if $\Lambda$ is amenable, then $H$ is also amenable, hence $\cR$ is amenable.
\end{proof}

\begin{proof}[Proof of Theorem \ref{model ergodic kernel dense subgroup}]
    Let $\rho : G \times \R^*_+ \curvearrowright (Y,\eta)$ be an ergodic nonsingular action. Thanks to Theorem \ref{model ergodic kernel dense subgroup measure preserving case}, we can take a measure preserving equivalence relation $\cR_0$ on a probability space $(X_0,\nu_0)$ and an ergodic cocycle $c_0 : \cR_0 \rightarrow G$. Let $\cT$ be an ergodic amenable nonsingular equivalence relation on a probability space $(T,\zeta)$ such that the Radon--Nikodym cocycle $\delta : \cT \rightarrow \R^*_+$ is ergodic, i.e.\ $\cT$ is of type $\III_1$. 
    
    On $(X,\nu)=(X_0 \times T \times Y, \nu_0 \otimes \zeta \otimes \eta)$ consider the induced nonsingular equivalence relation $\cR=(\cR_0 \times \cT)_{\rho,c_0 \times \delta}$. Define the cocycle 
    $$c : \cR \ni (x,t,y,x',t',y') \mapsto c_0(x,x') \in G.$$

    Let us compute the Mackey range of $c \times \omega$ where $\omega$ is the Radon--Nikodym cocycle of $\cR$ with respect to $\nu$.
    
     Let $\sigma : G \times \R^*_+ \times Y \rightarrow \R^*_+$ be the Radon--Nikodym cocycle of $\rho$. Then the cocycle $\omega : \cR \rightarrow \R^*_+$ is given by  $$ \omega(x,t,y,x',t',y') = \delta(t,t') \sigma(c_0(x,x'),\delta(t,t'),y').$$

    The product cocycle $c \times \omega : \cR \rightarrow G \times \R^*_+$ is given by
    $$ (c \times \omega)(x,t,y,x',t',y') =\left( c_0(x,x'), \delta(t,t') \sigma(c_0(x,x'),\delta(t,t'),y') \right).$$

    We want to compute the algebra $$\rL^\infty(X_0 \times T \times Y \times G \times \R^*_+)^{\cR_{c \times \omega}}$$ of all invariant functions for the skew product equivalence relation $\cR_{c \times \omega}$. 
    
    Consider the action $\kappa :  G \times \R^*_+ \curvearrowright Y \times G \times \R^*_+$ given by $$(g,\lambda) \cdot (y,h,s) = (\rho(g,\lambda) y, g h, \lambda \sigma(g,\lambda,y) s).$$
    Then we have the equality $$\cR_{c \times \omega} = (\cR_0 \times \cT)_{\kappa,c_0 \times \delta}.$$ 
    Since $c_0 \times \delta : \cR_0 \times \cT \rightarrow G \times \R^*_+$ is ergodic, we conclude by Lemma \ref{dense range induced equivalence relation} that $$\rL^\infty(X_0 \times T \times Y \times G \times \R^*_+)^{\cR_{c \times \omega}}=\rL^\infty(Y \times G \times \R^*_+)^{\kappa}.$$

     We can regard these $\kappa$-invariant functions in $\rL^\infty(Y \times G \times \R^*_+)^{\kappa}$ as $G$-equivariant functions in $\rL^\infty(G, \rL^\infty(Y)^{(1,\alpha)})$ where $\alpha=\rho|_{\R^*_+}$ and $G$ acts on $\rL^\infty(Y)^{(1,\alpha)}$ by the restriction of the Maharam extension of $\rho|_G$. The Mackey range of $c \times \omega$, which is the restriction of the right translation action of $G \times \R^*_+$ on $\rL^\infty(Y \times G \times \R^*_+)^{\kappa}$, is then identified with $\rho^{(1)}$. Therefore, if we replace $\rho$ by $\rho^{(-1)}$ in our construction, we obtain that the Mackey range of $c \times \omega$ is precisely $(\rho^{(-1)})^{(1)} \cong \rho$.

    Observe that if  $\cR_0$ is amenable, then $\cR$ is also amenable. If $c_0$ takes its values in a dense countable subgroup $\Lambda < G$, then so does $c$. 

    Suppose that $\alpha=\rho|_{\R^*_+}$ is ergodic and choose $c_0$ such that $\ker(c_0)$ is ergodic. Observe that $\ker(c)$ is the induced equivalence relation $(\ker(c_0) \times \cT)_{\alpha,1 \times \delta}$.
     Since the cocycle $1 \times \delta  : \ker(c_0) \times \cT \rightarrow \R^*_+$ is ergodic then, thanks to Lemma \ref{dense range induced equivalence relation}, we conclude that $\ker(c)$ is ergodic.
\end{proof}

\begin{proof}[Proof of Theorem \ref{ergodic kernel dense subgroup}] The theorem is trivial if $\cR$ is of type $\rm I$ because $c$ is a coboundary in this case and $G$ must be trivial. Assume now that $\cR$ is not of type $\rm I$, i.e.\ that it acts on a diffuse standard probability space $(X, \nu)$. Let $\omega : \mathcal R \to \R^*_+$ be the Radon--Nikodym cocycle associated with $\nu$. Denote by $\rho : G \times \R^*_+ \curvearrowright (Y, \eta)$ the Mackey range of the product cocycle $c \times \omega : \mathcal R \to G \times \R^*_+$. Since $c$ is ergodic we know that $\rho|_{\R^*_+} : \R^*_+ \curvearrowright \rL^\infty(Y, \eta)$ is ergodic.

By Theorem \ref{model ergodic kernel dense subgroup}, there exists an amenable ergodic equivalence relation $\mathcal S$ and an ergodic cocycle $d : \mathcal S \to G$ such that $d$ takes its values in $\Lambda$, $\ker(d) < \mathcal S$ is ergodic and the Mackey range of the product cocycle $d \times \delta : \mathcal S \to G \times \R^*_+$ is isomorphic to $\rho$ where $\delta : \cS \rightarrow \R^*_+$ is the Radon--Nikodym cocycle. Since the Krieger flow of $\mathcal S$ coincides with $\rho|_{\R^*_+} : \R^*_+ \curvearrowright \rL^\infty(Y, \eta)^G$, it follows that $\mathcal R$ and $\mathcal S$ are isomorphic when $\mathcal R$ if of type $\III$ and stably isomorphic when $\mathcal R$ is of type ${\rm II}$ \cite{CFW81, Kr75}. Actually, if $\cR$ is of type $\II_1$, we can choose $\mathcal S$ to be also of type $\II_1$ by Theorem \ref{model ergodic kernel dense subgroup measure preserving case}. If $\cR$ is of type $\II_\infty$, we can amplify $\mathcal S$ so that it becomes also of type $\II_\infty$. Hence in all cases, we obtain  an equivalence relation $\cS$ that is isomorphic to $\cR$. Thus, we may assume that $d$ is a cocycle defined on $\cR$.

We have constructed an ergodic cocycle $d : \mathcal R \to G$ that takes its values in $\Lambda$, such that $\ker(d) < \mathcal R$ is ergodic and such that the Mackey range of the product cocycle $d \times \omega : \mathcal R \to G \times \R^*_+$ is isomorphic to the Mackey range of $c \times \omega$.

By Theorem \ref{classification of cocycles}, there exists $\theta \in \Aut(\mathcal R)$ such that $c$ and $d \circ (\theta \times \theta)$ are cohomologous. Therefore, upon replacing $d$ by $d \circ (\theta \times \theta)$, we are done.
\end{proof}




\section{The cocycle perturbation theorem}
Our goal in this section is to prove Theorem \ref{cocycle perturbation flow} by using Theorem \ref{ergodic kernel dense subgroup}. In fact, we will prove the following more general version for arbitrary locally compact abelian groups.

\begin{theorem} \label{cocycle perturbation}
    Let $M$ be a factor with separable predual. Let $\alpha : G \curvearrowright M$ be an action of a second countable locally compact abelian group $G$. Suppose that $\alpha$ is almost periodic and that $M \rtimes_\alpha G$ is a factor. Then $\alpha$ is cocycle equivalent to an action $\beta : G \curvearrowright M$ such that $\beta$ is almost periodic, $M^\alpha \subset M^\beta$ and $M^\beta$ is a factor.

    More precisely, we can choose $\beta$ of the form $\beta_g=\Ad(v_g) \circ \alpha_g, \; g \in G$ for some continuous morphism $v : G \rightarrow \cU(\cZ(M^\alpha))$.
\end{theorem}

\begin{lemma} \label{central cocycle equivalence}
    Let $M$ be a von Neumann algebra and let $\alpha : G \curvearrowright M$ be a continuous action of a locally compact abelian group $G$. Let $v : G \rightarrow \cU(\cZ(M^\alpha))$ be a continuous morphism and define a new action $\beta : G \curvearrowright M$ by $\beta_g=\Ad(v_g) \circ \alpha_g$ for all $g \in G$. Then $M^\alpha \subset M^\beta$ and $\cZ(M^\beta) \subset \cZ(M^\alpha)$.
\end{lemma}
\begin{proof}
    It is clear that $M^\alpha \subset M^\beta$. Take $x \in \cZ(M^\beta)$. Since $v_g \in M^\beta$, $x$ commutes with $v_g$ hence $\alpha_g(x)=\beta_g(x)=x$ for all $g \in G$. Thus, $x \in M^\alpha$ and we conclude that $x \in M^\alpha \cap \cZ(M^\beta) \subset \cZ(M^\alpha)$.
\end{proof}

Let $\Gamma$ be a countable discrete group $\Gamma$, $(X, \nu)$ a standard probability space, $\theta : \Gamma \curvearrowright (X, \nu)$ any (not necessarily free) nonsingular action and $G$ any second countable locally compact group. A cocycle $c : \Gamma \times X \to G$ is a measurable map that satisfies $c(\gamma_1 \gamma_2,x)=c(\gamma_2,\gamma_1 \cdot x)c(\gamma_1,x)$ for all $\gamma_i \in \Gamma$ and almost every $x \in X$. We say that $c$ is a coboundary if it is of the form $c(\gamma,x)=f(\gamma \cdot x) f(x)^{-1}$ for some measurable map $f : X \rightarrow G$. This relation is denoted $c=\partial f$.

We denote by $\rZ^1(\theta, G)$ (resp.\! $\rB^1(\theta, G)$) the space of all measurable cocycles (resp.\! coboundaries) with values in $G$ with respect to the action $\theta : \Gamma \curvearrowright (X, \nu)$.

Theorem \ref{cocycle perturbation} will follow from the following theorem.

\begin{theorem}  \label{ergodic kernel dense subgroup for non free actions}
Let $G$ be a second countable locally compact abelian group. Let $\Gamma < G$ be a countable dense subgroup. Let $\theta : \Gamma \curvearrowright (X,\nu)$ be a (not necessarily free) nonsingular action. Let $c \in \rZ^1(\theta,\Gamma)$ be the tautological cocycle given by $c(\gamma,x)=\gamma$ for all $(\gamma,x) \in \Gamma \times X$.

Assume that the diagonal action $\theta \otimes \rho : \Gamma \curvearrowright (X \times G,\nu \otimes m_G)$ is ergodic where $\rho : \Gamma \curvearrowright G$ is the action by left multiplication. 

Then for every countable dense subgroup $\Lambda < G$ that contains $\ker(\theta)$, we can find a cocycle $d \in \rZ^1(\theta,\Lambda)$ that is cohomologous to $c$ inside $\rZ^1(\theta,G)$ and such that the subgroupoid $\ker(d) < \Gamma \ltimes_\theta X$ acts ergodically on $(X,\nu)$.
\end{theorem}

We will reduce this cohomology problem for actions that are not necessarily free to a cohomology problem for equivalence relations by using the following proposition.

\begin{proposition} \label{abelian groupoid}
    Let $\theta : \Gamma \curvearrowright (X,\nu)$ be an ergodic nonsingular action of a countable abelian group $\Gamma$. Let $N = \ker(\theta)$ and let $\cR$ be the orbit equivalence relation of $\theta$. Then the transformation groupoid associated to $\theta$ decomposes as a direct product $\Gamma \ltimes_\theta X \cong  N \times \cR$.
\end{proposition}
\begin{proof}
    Let $ p : \Gamma \ltimes_\theta X \rightarrow \cR$ be the quotient map. Then $\ker(p)=\ker(\theta) =N$ because $\Gamma$ is abelian and $\theta$ is ergodic. Since $\cR$ is hyperfinite \cite{CFW81}, we can find a transformation $T \in [\cR]$ acting freely on $X$ that generates $\cR$. We can then find a bisection $S \in [\Gamma \ltimes X]$ such that $p(S(x))=(T(x),x)$ for all $x \in X$. Then the map 
    $$(g, T^n(x),x) \mapsto g \cdot S^n(x)$$
    is a groupoid isomorphism from $N \times \cR$ onto $\Gamma \ltimes X$.
\end{proof}

Now we can use Theorem \ref{ergodic kernel dense subgroup} to prove Theorem \ref{ergodic kernel dense subgroup for non free actions}.

\begin{proof}[Proof of Theorem \ref{ergodic kernel dense subgroup for non free actions}]
Let $N=\ker( \theta)$. Thanks to Proposition \ref{abelian groupoid}, we can identify the transformation groupoid $\Gamma \ltimes X$ with $N \times \cR$ where $\cR$ is an ergodic hyperfinite equivalence relation in such a way that $c \in \rZ^1(\theta,\Gamma)$ is identified with a cocycle of the form $\id \times c_0$ for some cocycle $c_0 : \cR \rightarrow \Gamma$. 

Let $H=\overline{N} < G$. Let $\widetilde{c}_0 : \cR \rightarrow G/H$ be the image of $c_0$ by the quotient map $ p : G \rightarrow G/H$. The assumption that $\theta \otimes \rho$ is ergodic means that the cocycle $\widetilde{c}_0$ is ergodic. Observe that $\widetilde{\Lambda}=p(\Lambda)$ is dense in $G/H$. Therefore, by Theorem \ref{ergodic kernel dense subgroup}, we know that $\widetilde{c}_0$ is cohomologous to a cocycle $\widetilde{d}_0 : \cR \rightarrow G/H$ with values in $\widetilde{\Lambda}$ such that the relation $\cS = \ker{\widetilde{d}_0}$ is ergodic. Write $\widetilde{d}_0 = (\partial \widetilde{f}) \cdot \widetilde{c}_0 $ for some $\widetilde{f} : X \rightarrow G/H$. Let $f : X \rightarrow G$ be a measurable lift of $\widetilde{f}$. Consider the cocycle $d_0 : \cR \rightarrow \Lambda$ defined by the formula $d_0 = (\partial f) \cdot c_0$. Observe that the cocycle $d_0|_{\cS}$ takes its values in $H$. Since $N$ is dense in $H$, by applying Lemma \ref{dense subgroup} to the cocycle $d_0|_{\cS}$, we can modify the choice of the lift $f$ to ensure that $d_0|_{\cS}$ takes its values in $N$. By construction, $d_0$ is cohomologous to $c_0$ thus $d = \id \times d_0 : N \times \cR \rightarrow \Lambda$ is cohomologous to $c = \id \times c_0 : N \times \cR \rightarrow \Gamma$ inside $\rZ^1(\theta,G)$. Moreover, $\ker(d)$ contains $(d_0(x,y)^{-1}, x,y)$ for all $(x,y) \in \cS$. Since $\cS$ is ergodic, we conclude that $\ker(d)$ is also ergodic.
\end{proof}

\begin{proof}[Proof of Theorem \ref{cocycle perturbation}]
    We suppose first that $\alpha$ comes from a discrete decomposition. This means that $M=M^\alpha \rtimes_\theta \Gamma$ for some action $\theta : \Gamma \curvearrowright M^\alpha$ where $\Gamma < \widehat{G}$ is a countable dense subgroup, and $\alpha = \widehat{\theta} \circ \iota$ where $\widehat{\theta} : \widehat{\Gamma} \curvearrowright M$ is the dual action of $\theta$ and $\iota : G \rightarrow \widehat{\Gamma}$ is the dual inclusion of $\Gamma < \widehat{G}$.  Write $\mathcal Z(M^\alpha) = \rL^\infty(X, \nu)$, let $\eta : \Gamma \curvearrowright (X, \nu)$ be the restriction of $\theta$ to $\cZ(M^\alpha)$. Let $\rho : \Gamma \curvearrowright (\widehat{G}, m_{\widehat G})$ be the translation action.    We have $M \rtimes_\alpha G=(M^\alpha \ovt \rL^\infty(\widehat{G})) \rtimes_{\theta \otimes \rho} \Gamma$. Since $M \rtimes_\alpha G$ is a factor, we must have $(\cZ(M^\alpha) \ovt \rL^\infty(\widehat{G}))^{\eta \otimes \rho} = \C$. In other words, the diagonal action $\eta \otimes \rho : \Gamma \curvearrowright (X \times \widehat{G}, \nu \otimes m_{\widehat G})$ is ergodic.

  Consider the tautological cocycle $c \in \rZ^1(\eta,\widehat{G})$ defined by $c(\gamma,x)=\gamma$ for every $\gamma \in \Gamma$ and every $x \in X$. By Theorem \ref{ergodic kernel dense subgroup for non free actions}, we can find a cocycle $d \in \rZ^1(\eta,\widehat{G})$ with values in $\Gamma$ such that $d$ is cohomologous to $c$ and the groupoid $\ker d$ acts ergodically on $(X, \nu)$. Let $f : X \rightarrow \widehat{G}$ be a measurable map such that $d=(\partial f) \cdot c$ where $\partial f \in \rB^1(\eta,\widehat{G})$. Define a continuous morphism $v : G \rightarrow \mathcal{U}(\cZ(M^\alpha))$ by the formula $v_g = \langle f, g \rangle$
   for all $g \in G$. Then we have $$\langle (\partial f)(\gamma), g \rangle =v_g \eta_\gamma(v_g^*)$$ for all $\gamma \in \Gamma$ and all $g \in G$.
   
   Now, let $\beta : G \curvearrowright M $ be the action given by $\beta_g= \Ad(v_g) \circ \alpha_g$. We have, for all $\gamma \in \Gamma$ and all $g \in G$, 
   \begin{align*}
 \beta_g(u_\gamma) &=v_g \alpha_g(u_\gamma) v_g^* \\
 &=\langle \gamma,g \rangle v_g  u_\gamma v_g^* \\
 &= \langle \gamma, g \rangle v_g \eta_\gamma(v_g^*) u_\gamma \\
 &= \langle \gamma, g \rangle \langle \partial f(\gamma), g \rangle  u_\gamma.  
   \end{align*}
  Since $d= (\partial f) \cdot c$, we obtain $\beta_g(u_\gamma)= \langle d(\gamma),g \rangle u_\gamma$. 
   
   From this formula, we first see that $\beta$ is $\Gamma$-almost periodic. Secondly, we see that the groupoid subalgebra $\rL(\ker d) \subset \cZ(M^\alpha) \rtimes_{\eta} \Gamma$ is contained in $M^\beta$. Using Lemma \ref{central cocycle equivalence}, we conclude that 
    $$\cZ(M^\beta) \subset \rL(\ker d)' \cap \cZ(M^\alpha)=\C1$$ 
    because the groupoid $\ker d$ is ergodic. In other words, $M^\beta$ is a factor. This settles the case where $\alpha$ comes from a discrete decomposition.

    Now, in the general case we know that $\widetilde{\alpha}=\alpha \otimes \id$ on $\widetilde{M}=M \ovt \B(\ell^2)$ comes from a discrete decomposition. Therefore, we can find a continuous morphism $\widetilde{v} : G \rightarrow \cU(\cZ(\widetilde{M}^{\widetilde{\alpha}}))$ such that $\widetilde{M}^{\widetilde{\beta}}$ is a factor where $\widetilde{\beta}$ is the action given by $\widetilde{\beta}_g=\Ad(\widetilde{v}_g) \circ \widetilde{\alpha}_g$. Now, observe that $\widetilde{M}^{\widetilde{\alpha}}=M^\alpha \ovt \B(\ell^2)$, hence $\cZ(\widetilde{M}^{\widetilde{\alpha}})=\cZ(M^\alpha) \otimes \C1$. Thus, we can write $\widetilde{v}_g = v_g \otimes 1$ for some continuous morphism $v : G \rightarrow \cU(\cZ(M^\alpha))$, and $\widetilde{\beta}_g=\beta_g \otimes \id$ where $\beta_g=\Ad(v_g) \circ \alpha_g$. Finally, since $\widetilde{M}^{\widetilde{\beta}}=M^\beta \ovt \B(\ell^2)$, we conclude that $M^\beta$ is a factor.
\end{proof}

\section{The Rokhlin property for almost periodic actions}

Our goal in this section is to prove Theorem \ref{main almost periodic flow} and more generally, the following version for almost periodic actions of locally compact abelian groups.
\begin{theorem} \label{almost periodic action}
     Let $G$ be a second countable locally compact abelian group. Let $\alpha : G \curvearrowright R$ be an almost periodic action on the hyperfinite $\II_1$ factor such that $R \rtimes_\alpha G$ is a factor and $\alpha$ is outer. Then $\alpha$ has the Rokhlin property.
\end{theorem}

Thanks to Theorem \ref{cocycle perturbation}, it is sufficient to prove Theorem \ref{almost periodic action} under the assumption that $\alpha$ is prime. To deal with this case we combine the ideas of \cite[Proposition 2.5]{Ki95} and \cite[Theorem 6.12]{MT12}.

\begin{theorem}
     Let $G$ be a second countable locally compact abelian group. Let $\alpha : G \curvearrowright R$ be a continuous action on the hyperfinite $\II_1$ factor that is almost periodic, prime and outer. Then $\alpha$ has the Rokhlin property.
\end{theorem}
\begin{proof}
    Let $\Gamma < \widehat{G}$ be the point spectrum of $\alpha$ and write $\alpha=\kappa \circ \iota$ where $\kappa : \widehat{\Gamma} \curvearrowright R$ is a continuous action and $\iota : G \rightarrow \widehat{\Gamma}$ is the dual inclusion of $\Gamma < \widehat{G}$. Of course, we have $R^\kappa=R^\alpha$, hence $\kappa$ is a prime action. We distinguish three cases.

    \textbf{Case 1: $\alpha$ is ergodic, i.e.\! $R^\alpha=\C$.} Thanks to \cite{OPT79}, we know that $\kappa$ is isomorphic to the dual action on a twisted group von Neumann algebra $\rL_\Omega(\Gamma)$ for some symplectic bicharacter $\Omega : \Gamma \times \Gamma \rightarrow \T$. This means that $R$ is generated by unitaries $(u_g)_{g \in \Gamma}$ such that $\tau(u_a)=0$ for all $a \neq 1$, $u_au_bu_a^*u_b^*=\Omega(a,b)$ for all $a,b \in \Gamma$ and $\kappa_g(u_a)=\langle g, a \rangle u_a$ for all $g \in \widehat \Gamma$ and $a \in \Gamma$.
    
    Let $\varphi : \Gamma \rightarrow \widehat{\Gamma}$ be the unique morphism that satisfies $\Omega(a,b)=\langle \varphi(a),b \rangle$ for all $a,b \in \Gamma$. We claim that the subgroup $\Phi = \{ (\varphi(a),a) \in \widehat{\Gamma} \times \widehat{G} \mid a \in \Gamma \}$ is dense in $\widehat{\Gamma} \times \widehat{G}$. Indeed, by duality, it is enough to show that the orthogonal subgroup $\Phi^\perp < \Gamma \times G$ is trivial. Take $(b,g) \in \Phi^\perp$. This means that for all $a \in \Gamma$, we have $\langle a,g \rangle \langle \varphi(a),b \rangle=1$, or equivalently  $\langle a,g \rangle =\overline{\langle \varphi(a),b \rangle} =\overline{\Omega(a,b)}=\Omega(b,a)$. Therefore $$\alpha_g(u_a)=\langle a,g \rangle u_a = \Omega(b,a) u_a = u_bu_au_b^*$$
    for all $a \in \Gamma$. This shows that $\alpha_g = \Ad(u_b)$. Since $\alpha$ is outer and $R$ is a factor, this is only possible if $g=1$ and $b=1$. We conclude that $\Phi^\perp$ is trivial, hence that $\Phi$ is dense in $\widehat{\Gamma} \times \widehat{G}$ as claimed.

    Now, take $p \in \widehat{G}$ and take a sequence $(a_n)_{n \in \N}$ in $\Gamma$ such that $$\lim_{n \to \infty} (\varphi(a_n),a_n)=(1,p).$$ Then $(u_{a_n})_{n \in \N}$ is a central sequence in $R$ satisfying $$\lim_{n \to \infty} \| \alpha_g(u_{a_n}) - \langle p,g\rangle u_{a_n}\|_2 = 0$$ for all $g \in G$, uniformly on compact subsets of $G$. This shows that $\alpha$ has the Rokhlin property.

    \textbf{Case 2: $R^\alpha$ is of type $\I$.} Let $Q=(R^\alpha)' \cap R$ and let $\alpha|_Q : G \curvearrowright Q$ denote the restricted action. Then $R = R^\alpha \ovt Q$ and $\alpha = \id \otimes \alpha|_Q$. Since $\alpha|_Q$ is almost periodic, ergodic and outer, it has the Rokhlin property by the previous case. Any Rokhlin sequence for $\alpha|_Q$ in $Q$ is automatically asymptotically central in $R = R^\alpha \ovt Q$ and satisfies the same eigenvalue relation, so $\alpha$ inherits the Rokhlin property.
 
    \textbf{Case 3: $R^\alpha$ is of type $\II_1$.} 
    In this case, we know from \cite[§3]{Jo82} that $\kappa$ is a dual action, that is $R=N \rtimes_\theta \Gamma$ for some action $\theta : \Gamma \curvearrowright N=R^\alpha$ and $\kappa= \widehat{\theta} : \widehat{\Gamma} \curvearrowright R$ is the dual action of $\theta$. Let $H=\{ h \in \Gamma \mid \theta_h \text{ is inner} \}$. Let $\varphi : \Gamma \rightarrow \widehat{H}$ be the morphism defined by $\langle \varphi(a), h \rangle =\theta_a(v)v^*$ where $v \in \cU(N)$ is any unitary that satisfies $\Ad(v)=\theta_h$. 
    
    We claim that the subgroup $\Phi = \{ (\varphi(a),a) \in \widehat{H} \times \widehat{G} \mid a \in \Gamma \}$ is dense in $\widehat{H} \times \widehat{G}$. For this, it is enough to show that the orthogonal subgroup $\Phi^\perp < H \times G$ is trivial. Take $(h,g) \in \Phi^\perp$. This means that $\langle a,g \rangle  = \overline{\langle  \varphi(a),h \rangle}$ for all $a \in \Gamma$. Take $v \in \cU(N)$ such that $\Ad(v)=\theta_h$. Let $w=vu_h^* \in \cU(R)$. Then $w \in N' \cap R$ and for every $a \in \Gamma$, we have 
    $$ wu_aw^*=vu_av^*=v\theta_a(v^*)u_ a =\overline{\langle \varphi(a),h \rangle} u_a=\langle a,g \rangle u_a.$$
    This shows that $\alpha_g=\Ad(w)$. Since $\alpha_g$ is outer, we conclude that $g=1$ and $w \in \T$, which means that $h=1$. We conclude that $\Phi^\perp$ is trivial, hence that $\Phi$ is dense in $\widehat{H} \times \widehat{G}$, as claimed. 

    Now, take $p \in \widehat{G}$ and take a sequence $a_n \in \Gamma$ such that $$\lim_{n \to \infty} (\varphi(a_n),a_n)=(1,p).$$ Take  a nonprincipal ultrafilter $\omega$ on $\N$ and let $u=(u_{a_n})^\omega$. Observe that $u \in R^\omega_\alpha$ (equicontinuous part) and that $\alpha^\omega_g(u)=\langle p,g \rangle u$ for all $g \in G$. 
    Since $N$ is hyperfinite, every $\theta_{a_n}$ is approximately inner. Thus, we can find a sequence of unitaries $v_n \in \cU(N)$ such that $\theta_{a_n} \circ \Ad(v_n)^{-1} \to \id$. Let $v=(v_n)^\omega \in N^\omega$. Then $uv^* \in N' \cap R^\omega_\alpha$. Consider the $1$-cocycle $c : \Gamma \rightarrow \cU(N_\omega)$ defined by $c(a)=\theta_a^\omega(v)v^*$. Take $h \in H$ and $w \in \cU(N)$ such that $\theta_h=\Ad(w)$. Since $\varphi(a_n) \to 1$, we know that $uw=wu$. Since $uv^* \in N' \cap R^\omega_\alpha$ also commutes with $w$, then so does $v$. Hence $c(h)=wvw^*v^*=1$. Thus the $1$-cocycle $c : \Gamma \rightarrow \cU(N_\omega)$ is trivial on $H$ and we can view it as a $1$-cocycle from $\Gamma/H \rightarrow \cU(N_\omega)$. Since the action of $\Gamma/H$ on $N_\omega$ is free and liftable, we know by Ocneanu's 1-cocycle vanishing theorem \cite[Proposition 7.2]{Oc85} that there exists $z \in \cU(N_\omega)$ such that $c(a) = \theta_a^{\omega}(z)z^*$ for every $a \in \Gamma$. This means that $\theta_a^\omega(vz^*)=vz^*$ for all $a \in \Gamma$. It then follows that $uv^*z$ commutes with $u_a$ for every $a \in \Gamma$. Since $uv^*z$ also commutes with $N$ by construction, we conclude that $uv^*z \in  R' \cap R^\omega$. We thus obtained a unitary $U=uv^*z \in \mathcal U(R_{\alpha, \omega})$ such that $$\alpha^\omega_g(U)=\alpha^\omega_g(uv^*z)=\alpha_g^\omega(u)v^*z=\langle g,p \rangle uv^*z=\langle g,p \rangle U$$
    for all $g \in G$. This shows that $\alpha$ has the Rokhlin property.
\end{proof}

We conclude this section by observing that for any almost periodic action $\alpha : G \curvearrowright M$ of a locally compact abelian group on an arbitrary factor, $\alpha$ is strictly outer if and only if $\alpha$ is outer and $M \rtimes_\alpha G$ is a factor. More generally, we can describe the relative commutant $M' \cap (M \rtimes_\alpha G)$ when $\alpha$ is not outer.

\begin{theorem}
    Let $M$ be a factor. Let $\alpha : G \curvearrowright M$ be an action of a locally compact abelian group $G$. Suppose that $\alpha$ is almost periodic and let $E^\alpha : M \rightarrow M^\alpha$ be the canonical $\alpha$-invariant faithful normal conditional expectation. Let $H=\{g \in G \mid \alpha_g \in \Inn(M)\}$ and for every $h \in H$, write $\alpha_h=\Ad(v_h)$ for some $v_h \in \cU(M)$. Let $\chi : H \rightarrow \widehat{G}$ be the morphism defined by $\alpha_g(v_h)=\langle g, \chi(h) \rangle v_h$ for all $g \in G$ and $h \in H$.

    Assume that $M \rtimes G$ is a factor. Then the following hold:
    \begin{enumerate}
        \item $\chi$ is injective.
        \item The unitaries $(v_h)_{h \in H}$ are pairwise orthogonal in $M$ with respect to $E^\alpha$. In particular, if $M$ has separable predual, then $H$ is at most countable.
        \item $M' \cap (M \rtimes_\alpha G)$ is generated by the unitaries $(v_h^*u_h)_{h \in H}$. In particular, $\alpha$ is outer if and only if $M' \cap (M \rtimes_\alpha G)=\C1$.
    \end{enumerate}

\end{theorem}
\begin{proof}
    (1) Take $h \in \ker(\chi)$. Then $v_h \in M^\alpha$. Since $\alpha_h=\Ad(v_h)$, we know that $v_h^*u_h \in M' \cap (M \rtimes_\alpha G)$. Since $v_h \in M^\alpha$, $v_h^*u_h$ commutes with $u_g$ for every $g \in G$. Thus, $v_h^*u_h$ is in the center of $M \rtimes_\alpha G$, which is trivial by assumption. We conclude that $h=1$.

    (2) Take $h,k \in H$ with $h \neq k$. We have $\alpha_g(v_h^*v_k)=\langle g, \chi(h^{-1}k) \rangle v_h^*v_k$ for every $g \in G$. Hence $E^\alpha(v_h^*v_k)=\langle g, \chi(h^{-1}k) \rangle E^\alpha(v_h^*v_k)$ for all $g \in G$. Since $\chi$ is injective, we know that $\chi(h^{-1}k)$ is a nontrivial character in $\widehat{G}$ and we conclude that $E^\alpha(v_h^*v_k)=0$.

    (3) Let $\Aut(M,\alpha)=\{ \theta \in \Aut(M) \mid \theta \text{ commutes with } \alpha \}$. Observe that we have a natural continuous homomorphism $\pi : \Aut(M,\alpha) \rightarrow \Aut(M \rtimes_\alpha G)$ such that $\pi(\theta)|_M=\theta$ and $\pi(\theta)$ fixes $\rL(G)$ for all $\theta \in \Aut(M,\alpha)$. Moreover, $\pi(\alpha_g)=\Ad(u_g)$ for all $g \in G$. Since $\alpha$ is almost periodic, it follows that the inner action $\pi \circ \alpha$ is also almost periodic. Let $\beta : G \curvearrowright M' \cap (M \rtimes_\alpha G)$ be the restriction of $\pi \circ \alpha$ to $M' \cap (M \rtimes_\alpha G)$. Then $\beta$ is almost periodic. Since $M \rtimes_\alpha G$ is a factor, we also know that $\beta$ is ergodic. Therefore, $M' \cap (M \rtimes_\alpha G)$ is generated by unitaries $w_p$ for some $p \in \widehat{G}$ such that $\alpha_g(w_p)=\langle p, g \rangle w_p$.

    This means that $\Ad(w_p)(u_g)=\langle p, g \rangle u_g$ for all $g \in G$. Hence $\widehat{\alpha}_p=\Ad(w_p)$ where $\widehat{\alpha} : \widehat{G} \curvearrowright M \rtimes_\alpha G$ is the dual action. Since $M \rtimes_\alpha G$ is a factor, there exists a nontrivial $g \in G$ such that $\widehat{\alpha}_q(w_p)=\langle q,g \rangle w_p$ for all $q \in \widehat{G}$. This means that $v_g=u_gw_p^*$ is fixed by $\widehat{\alpha}$, hence $v_g \in M$. Since $w_p=v_g^*u_g$ commutes with $M$, we conclude that $\alpha_g=\Ad(v_g)$, hence $g \in H$. 
\end{proof}

We note that the classification of almost periodic flows $\alpha : G \curvearrowright R$ on the hyperfinite $\II_1$ factor is not complete since the case where $R \rtimes_\alpha G$ is a factor and $\alpha$ is not outer is not covered by our main theorem. This case has no analogue in the classification of amenable type $\III$ factors. We make the following conjecture.

\begin{conjecture}
    Let $\alpha : G \curvearrowright R$ be an almost periodic action on the hyperfinite $\II_1$ factor such that $R \rtimes_\alpha G$ is a factor. Then $\alpha$ is classified up to cocycle conjugacy by the characteristic invariant $(H,\chi)$ where $H=\{h \in G \mid \alpha_h \in \Inn(M) \}$ and $\chi : G \times H \rightarrow \T$ is the bicharacter defined by $\chi(g,h)=\alpha_g(v)v^*$ for every $g \in G$, $h \in H$ and $v \in \cU(R)$ such that $\Ad(v)=\alpha_h$.
\end{conjecture}

\section{Almost periodic factors of type $\III_1$}

The goal of this section is to prove Theorem \ref{main extremal}. Let $M$ be a von Neumann algebra with separable predual and $\varphi, \psi \in \mathcal P(M)$. Following \cite[Chapter ${\rm VIII}$]{Ta03}, we say that $\varphi$ and $\psi$ are {\em commuting} if $\psi = \psi \circ \sigma_t^\varphi$ for every $t \in \R$. This relation is actually symmetric. 

Recall that if $h$ is a nonsingular positive selfadjoint operator affiliated with $M_\varphi$, then we can define a new weight $\varphi(h \, \cdot) \in \cP(M)$. This weight $\varphi(h \, \cdot)$ is characterized by the fact that its Connes--Radon--Nikodym derivative with respect to $\varphi$ is the map $\R \to \cU(M) : t \mapsto h^{\ri t}$ (see \cite[Lemma ${\rm VIII}$.2.8]{Ta03}).

We record the following well-known fact (see \cite[Corollary ${\rm VIII}$.3.6]{Ta03}).

\begin{lemma}\label{lem-commuting}
Let $M$ be a von Neumann algebra with separable predual and $\varphi, \psi \in \mathcal P(M)$. The following assertions are equivalent:
\begin{enumerate}
\item The weights $\varphi$ and $\psi$ are commuting.

\item There exists a nonsingular positive selfadjoint operator $h$ affiliated with $M_\varphi$ such that $\psi = \varphi(h \, \cdot)$.
\end{enumerate}
In that case, we have $M_\varphi \subset M_\psi$ if and only if $h$ is affiliated with $\mathcal Z(M_\varphi)$.
\end{lemma}

We prove the following more precise version of Theorem \ref{main extremal}.

\begin{theorem} \label{precise extremal}
    Let $M$ be a factor of type $\III_1$ with separable predual. Then for every almost periodic weight $\varphi \in \cP(M)$, there exists an extremal almost periodic weight $\psi \in \cP(M)$ such that $\psi$ and $\varphi$ are commuting and $M_\varphi \subset M_\psi$. 
\end{theorem}
\begin{proof}
    Let $\alpha =\sigma^\varphi : \R \curvearrowright M$. Then $M^\alpha=M_\varphi$. Since $M$ is of type $\III_1$, we know that $M \rtimes_\alpha \R \cong c(M)$ is a factor. By Theorem \ref{cocycle perturbation}, we can thus find a one-parameter unitary group $u : \R \to \mathcal U(\cZ(M_\varphi)) : t \mapsto u_t$ such that the action $\beta : t \mapsto \Ad(u_t) \circ \alpha_t$ is almost periodic and $M^\beta$ is a factor. By the Connes--Radon--Nikodym cocycle theorem \cite{Co72}, we know that $\beta=\sigma^\psi$ for some $\psi \in \cP(M)$. Then $\psi$ is almost periodic, it commutes with $\varphi$ and $M_\psi=M^\beta$ is a factor that contains $M_\varphi=M^\alpha$.
\end{proof}

As a consequence of Theorem \ref{precise extremal}, we obtain the following new characterization of extremal almost periodic weights.

\begin{corollary} \label{cor extremal}
 Let $M$ be a type $\III_1$ factor with separable predual and $\varphi \in \mathcal P(M)$ an almost periodic weight. The following assertions are equivalent: 
 \begin{enumerate}
 \item $\varphi$ is extremal.
 \item $M_\varphi \subset M$ is a maximal semifinite von Neumann subalgebra.
 \end{enumerate}
\end{corollary}

We can deduce Corollary \ref{cor extremal} from Theorem \ref{precise extremal}.

\begin{proof}[Proof of Corollary \ref{cor extremal}]
We use the equivalence proved in Lemma \ref{lem-commuting}.  

$(1) \Rightarrow (2)$ Suppose that $\varphi \in \mathcal P(M)$ is an extremal almost periodic weight and denote by $\R < K$ a dense embedding into a second countable compact group together with the corresponding continuous extension $\alpha : K \curvearrowright M$ of the modular automorphism group $\sigma^\varphi : \R \curvearrowright M$. Then we have $M_\varphi = M^\alpha$. Let $M_\varphi \subset P \subset M$ be an intermediate semifinite von Neumann subalgebra. Since $(M^\alpha)' \cap M = (M_\varphi)' \cap M = \C 1$, \cite[Theorem 3.15]{ILP96} implies that $P \subset M$ is the range of a faithful normal conditional expectation and thus $P \subset M$ is necessarily globally invariant under $\sigma^\varphi$. Denote by $\tau \in \mathcal P(P)$ a trace on $P$ and by $h$ the unique nonsingular positive selfadjoint operator affiliated with $P$ such that $\varphi|_P = \tau(h \, \cdot)$. Then for every $t \in \R$ and every $x \in M_\varphi \subset P$, we have $x = \sigma_t^\varphi(x) = h^{\ri t} x h^{- \ri t}$. Since $(M_\varphi)' \cap M = \C 1$, we have that $h = \lambda 1$ for $\lambda > 0$ and so $M_\varphi = P$.

$(2) \Rightarrow (1)$ Suppose that $M_\varphi \subset M$ is a maximal semifinite von Neumann subalgebra. By Theorem \ref{precise extremal}, there exists an extremal almost periodic weight $\psi \in \cP(M)$ such that $\psi$ and $\varphi$ are commuting and $M_\varphi \subset M_\psi$. By maximality, we obtain $M_\varphi = M_\psi$. Then we have $(M_\varphi)' \cap M = (M_\psi)' \cap M = \C 1$ and so $\varphi$ is extremal. 
\end{proof}

\bibliographystyle{plain}

\end{document}